\documentclass{b-math}
\usepackage{amsmath,amssymb,amscd,amsthm,verbatim,alltt,amsfonts,array}
\usepackage[english]{babel}
\usepackage{latexsym}
\usepackage{euscript}
\usepackage{graphicx}
\usepackage[all]{xy}

\newcommand{\salt}{\vspace*{2.5mm}}     

\newcommand{\ap}{@}
\catcode`\@=13
\newcommand{@}[1]{\ifmmode{}^\prime\noexpand#1%
 \else\ifx#1a\u{a}%
 \else\ifx#1i\^{\i}%
 \else\ifx#1s\c{s}%
 \else\ifx#1t\c{t}%
 \else\ifx#1A\u{A}%
 \else\ifx#1I\^{I}%
 \else\ifx#1S\c{S}%
 \else\ifx#1T\c{T}%
 \else\ifx#1"\ap%
    \fi\fi\fi\fi\fi\fi\fi\fi\fi%
\fi}

\font\sc=cmcsc10

\newcommand{\eqnoone}  
   {}
\newcommand{\eqnotwo}  
   {}

\newcounter{alf}

\newcommand{\adresa}[1]{\par\vspace*{-11pt}
                        \begin{flushright}
                        {\small 
                        #1}
                        \end{flushright}
                        }

\newtheorem{theorem}{Theorem}
\theoremstyle{plain}

\newtheorem{corollary}{Corollary}

\newtheorem{lemma}{Lemma}

\newtheorem{proposition}{Proposition}

\theoremstyle{definition}
\newtheorem{exm}[theorem]{Example}
\newtheorem{rem}[theorem]{Remark}

\numberwithin{equation}{section}

\newcommand{\N}{\mathbb{N}}
\newcommand{\Z}{\mathbb{Z}}
\newcommand{\Q}{\mathbb{Q}}

\newcommand{\C}{\mathbb{C}}

\newcommand{\PP}{\mathbb{P}}

\newcommand{\al}{\alpha}

\newcommand{\RR}{{\mathcal R}}
\newcommand{\VV}{\mathcal{V}}

\newcommand{\OO}{\mathcal{O}}

\renewcommand{\sl}{{\mathfrak{sl}}}
\newcommand{\m}{{\mathfrak{m}}}
\newcommand{\n}{{\mathfrak{n}}}

\DeclareMathOperator{\im}{im}

\DeclareMathOperator{\ab}{{ab}}

\DeclareMathOperator{\Hom}{{Hom}}

\DeclareMathOperator{\ann}{{ann}}

\DeclareMathOperator{\supp}{supp}
\DeclareMathOperator{\specm}{Spec}

\DeclareMathOperator{\reg}{reg}

\newcommand{\surj}{\twoheadrightarrow}
\newcommand{\inj}{\hookrightarrow}

\newcommand{\imp}{\:\Longrightarrow\:}
\newcommand{\same}{\:\Longleftrightarrow\:}

\def\dot{\mathchar"013A}  
\newcommand{\hdot}{{\raise1pt\hbox to0.35em{\Large $\dot$\!}}} 

\newcommand{\cdga}{\ensuremath{\mathsf{cdga}}}

\begin{document}

\title{Algebraic models, Alexander-type invariants, and Green--Lazarsfeld sets}

\author{Alexandru Dimca, Stefan Papadima, Alexandru Suciu}

\date{July 27, 2014}

\pagestyle{myheadings}

\markboth{A. Dimca, S. Papadima, A. Suciu}%
{Algebraic models, Alexander-type invariants, and Green--Lazarsfeld sets}

\maketitle 

\begin{abstract}
We relate the geometry of the resonance varieties
associated to a commutative differential graded algebra 
model of a space to the finiteness properties of the 
completions of its Alexander-type invariants. We also 
describe in simple algebraic terms the non-translated 
components of the degree-one characteristic 
varieties for a class of non-proper complex manifolds. 
\end{abstract}

\begin{quotation}
\noindent{\bf Key Words}: {Resonance varieties, characteristic 
varieties, Alexander invariants, completion, Gysin models, intersection form.\\
}

Se\-con\-dary 14F35, 20J05.  
\end{quotation}

\thispagestyle{empty} 

\section{Introduction}
\label{sect:intro}

Throughout, $X$ will denote a connected CW-complex.  
Without loss of generality, we may assume 
that $X$ has a single $0$-cell.  Let  $\pi=\pi_1(X,x_0)$ 
be the fundamental group of our space based at this point, 
and let $\pi_{\ab}=H_1(X,\Z)$ be the abelianization of $\pi$. 

An old idea, going back to J.W.~Alexander's definition of 
his eponymous knot polynomial, is to consider the  homology 
groups of the universal abelian cover $X^{\ab}$, 
viewed as modules over the group-algebra $\C[\pi_{\ab}]$. 
The resulting Alexander invariants, $H_i(X^{\ab}, \C)$, 
play an important role in low-dimensional topology and 
geometry.  

Let $I$ be the augmentation ideal of $\C[\pi_{\ab}]$, and let 
$\widehat{H}_i(X^{\ab}, \C)$ be the completion of the $i$-th 
Alexander invariant with respect to the $I$-adic filtration. 
The first main result of this note is about the finiteness properties of 
these modules (as well as their natural generalizations), 
when viewed as complex vector spaces. 

Our approach is via commutative differential graded algebras
(for short, $\cdga$'s) $A^{\hdot}$, and their resonance varieties 
$\RR^i_r (A)$, which sit inside $H^1(A)$. We say that $A$ is a $q$-model 
for $X$ if $A$ has the same Sullivan $q$-minimal model as the de 
Rham $\cdga$ of $X$ \cite{Su}.  When the space $X$ has a 
$q$-model $A$ with good finiteness properties (i.e., when
$A$ is $q$-finite in the sense explained in Section \ref{sect:res}), 
we relate the finite-dimensionality of the completions of the Alexander 
invariants of $X$ to the geometry of the resonance varieties of $A$, as follows. 

\begin{theorem}
\label{thm:intro} 
Let $X$ be a connected CW-complex with finite $q$-skeleton. 
Suppose $X$ admits a $q$-finite $q$-model $A$. Then 
the complex vector space $\bigoplus_{i\le q} \widehat{H}_{i} (X^{\ab}, \C)$ 
is finite-dimensional if and only if $0$ is an isolated point in  
the variety $\bigcup_{i\le q} \RR^i_1 (A)$.
\end{theorem}

A particularly interesting class of spaces to which Theorem \ref{thm:intro} 
applies consists of smooth, connected, quasi-projective varieties, for which 
we may take as a suitable $\cdga$ model the Gysin model constructed 
by Morgan in \cite{M}.   Our theorem extends results from both \cite{DP-ann} 
and \cite{PS-mrl}, where only the formal case was considered, i.e., the case 
when $A$ may be taken to be the cohomology ring $H^{\hdot}(X, \C)$
with zero differential, and $\RR^i_r (X):=\RR^i_r (A)$ are the usual 
resonance varieties of $X$.  

The classical theory of Alexander polynomials of knots and links 
has a vast and fruitful generalization.  Given a connected, finite-type 
CW-complex $X$, the characteristic varieties $\VV^i_r(X)$ are certain 
algebraic subsets of the character torus $H^1(X,\C^*)$, defined as the 
jump loci for cohomology with coefficients in rank $1$ local systems 
on $X$.  (In the context of K\"ahler geometry, these varieties are also 
known as the Green--Lazarsfeld sets \cite{GL87, GL91}.) 
In degree $i=1$, the characteristic varieties depend only 
on the fundamental group $\pi=\pi_1(X)$, and provide 
powerful invariants for this group.  When $X$ is a 
quasi-projective manifold as above, the set 
$\VV^1_1(\pi)$ controls fibrations 
of $X$ onto curves of general type, cf.~Arapura \cite{A}.

Let $X$ be a connected, quasi-projective manifold which is not a curve; 
that is, $X= \overline{X}\setminus D$, where $\overline{X}$ is a 
projective manifold of dimension at least $2$, and $D$ is a 
union of smooth divisors with normal crossings. 
Let $D=\bigcup_{k=1}^m D^k$ be the decomposition of $D$ into {\it connected}\/
components.  For each index $k$, let $I_k$ be the intersection matrix 
associated to the divisor $D^k$; more specifically, $I_k=(D^k_i \cdot D^k_j)$, 
where $D^k_i$ and $D^k_j$ run through the irreducible components of $D^k$.
Denoting by $U_{(u)}$ the germ of the analytic set $U$ at a point $u$, 
we may state our second main result, as follows. 

\begin{theorem}
\label{thm2:intro} 
Assume that the intersection matrix $I_k$ is either negative or 
positive definite, for each $k=1,\ldots , m$.  Then there is an analytic germ 
isomorphism $H^1(X,\C^*) _{(1)} \cong H^1(\overline{X},\C)_{(0)}$
which induces  analytic germ isomorphisms 
$\VV^1_r(X)_{(1)} \cong\RR^1_r(\overline{X})_{(0)}$, for all $r\ge 0$.
\end{theorem}

In \cite{DP-ccm}, the germs $\VV^i_r(X)_{(1)}$ were identified with 
the germs $\RR^i_r(A)_{(0)}$, with no additional assumptions on $X$, 
by choosing $A$ to be the Gysin model associated to the 
above compactification of $X$. The main point in Theorem \ref{thm2:intro}
is the replacement of the Gysin model by the smaller, much simpler 
sub-$\cdga$ $(H^{\hdot}(\overline{X}, \C), d=0)$.
This theorem extends similar results from \cite{DPS, DPSquasi}.

The unifying idea behind both Theorem \ref{thm:intro} and 
Theorem \ref{thm2:intro} is the description of germs of 
characteristic varieties of spaces in terms of suitable 
$\cdga$ models, established in \cite{DP-ccm} and further 
elaborated in \cite{BW, MPPS}.  

Each of these theorems considers some rather intricate objects, 
namely, the Alexander invariants and the characteristic varieties 
of a reasonably nice space, and manages to relate these 
objects (upon completion, or by passing to germs at the origin) 
to some simpler objects, namely, the resonance varieties of 
an appropriate $\cdga$. 

\section{Some commutative algebra}
\label{sect:completion}

Let $R$ be a Noetherian ring, $\m \subset R$ a maximal ideal, 
and $M$ a finitely generated $R$-module. We will denote by 
$\widehat{M}$ the $\m$-adic completion of $M$, and view it 
as a module over $\widehat{R}$, the $\m$-adic completion of $R$. 

Let $\m_0=\m\cdot R_\m$ 
be the maximal ideal of the localized ring $R_\m=S^{-1}R$,
where $S=R \setminus \m$. The following result says that completion 
with respect to $\m$ corresponds to `analytic localization', and 
hence is stronger than algebraic localization.

\begin{proposition}
\label{prop:p1} 
The natural morphism $M \to M_{\m}=S^{-1}M$ induces an isomorphism 
$\widehat M \to \widehat {M_{\m}}$, where the first completion is with 
respect to $\m$, and the second completion is with respect to $\m_0$.
\end{proposition}

\begin{proof}
Applying Theorem 7.2 from \cite{Ei}, we see 
that $\widehat M = M \otimes_R \widehat R$ 
and $\widehat M_\m = M_\m \otimes_{R_\m} \widehat R_\m$. 
Since $M_\m=M \otimes_R  R_\m$, we also have 
$\widehat M_\m = M \otimes_{R} \widehat R_\m$.  
Hence, it is enough to consider only the case $M=R$.

To prove that $\widehat R \cong \widehat {R_\m}$, 
it is enough to show that there are isomorphisms
\[
R/\m^k \to R_\m/\m_0^k
\]
for all $k \in \N$. Using Proposition 2.5 from \cite{Ei}, 
we have that 
\[
R_\m/\m_0^k \cong S^{-1}R/S^{-1}\m^k \cong S^{-1}(R/\m^k).
\]

We claim that any element $s \in S$ becomes invertible in $R/\m^k $. 
To see this, let $I=\m^k+(s)$.  This ideal cannot be contained in a 
maximal ideal $\n$, since $\m^k \subset I \subset \n$ implies, 
by taking radicals, that $\m \subset \n$, and hence $\m=\n$. 
This is a contradiction, since $s \in \n \setminus \m$. It follows 
that $I=R$ and hence there are elements $x \in \m^k$ and $t \in R$
such that $x+st=1$.

The claim shows that $S^{-1}(R/\m^k) \cong R/\m^k$, thereby 
concluding the proof.
\end{proof}

Assume now in addition that $R$ is a finitely generated $\C$-algebra, 
and let  $\specm (R)$ be its maximal spectrum. 
The {\em support}\/ of the $R$-module $M$, denoted 
$\supp(M)$ is the subvariety of $\specm (R)$ defined by 
the annihilator ideal  of $M$ in $R$, denoted $\ann(M)$. 

\begin{proposition}
\label{prop:isol}
Let $M$ be a finitely-generated $R$-module, let  $\m$ be a 
maximal ideal in $R$, and let $\widehat{M}$ be the $\m$-adic 
completion of $M$.  
Then, the following conditions are equivalent.
\begin{enumerate}
\item \label{f1}
The $\C$-vector space $\widehat{M}$ is finite-dimensional. 
\item \label{f2}
The point $\m$ is isolated with respect to $\supp (M)$.
\end{enumerate}
\end{proposition}

\begin{proof}
Suppose first that $\widehat M=0$.  This condition is equivalent to 
$M_\m=0$, that is, $\m\notin \supp(M)$. 

Suppose now that $\widehat M \ne 0$.
By Proposition \ref{prop:p1}, $\dim_{\C}\widehat M <\infty$ if and only if 
$\dim_{\C}\widehat M_\m <\infty$.  In turn, the second condition 
is equivalent to saying that the module $\widehat M_\m$ has finite length; 
in other words, its Krull dimension is zero, see \cite[Proposition 10.8]{Ei}. 
By \cite[Corollary 12.5]{Ei},  
this is equivalent to the Krull dimension of $M_\m$ being equal to zero.
By \cite[Corollary 2.18]{Ei}, the last condition is equivalent to the fact that
$\{ \m \}$ is an irreducible component of $\supp (M)$.
\end{proof}
  
\section{Resonance varieties of a $\cdga$}
\label{sect:res}

Let $A=(A^{\hdot},d)$ be a commutative, differential 
graded algebra ($\cdga$) over the field of 
complex numbers.  That is to say,  $A=\bigoplus_{i\ge 0} A^i$ 
is a graded $\C$-vector space, endowed with a 
graded-commutative multiplication 
map $\cdot\colon A^i \otimes A^j \to A^{i+j}$ and a differential 
$d\colon A^i\to A^{i+1}$ satisfying the graded Leibnitz rule. 

We say that $A$ is $q$-finite, for some $q\ge 1$, 
if $A^0=\C$ and $A^i$ is finite-dimensional, for each $i\le q$.
Furthermore, we say that two $\cdga$'s $A$ and $B$ have 
the same $q$-type 
if there is a zig-zag of $\cdga$ maps from one to the other, 
inducing isomorphisms in cohomology in degree up to $q$, 
and a monomorphism in cohomology in degree $q+1$.  

For each cohomology class $\omega\in H^1(A)$, 
we make $A$ into a cochain complex, 
\begin{equation}
\label{eq:aomoto}
\xymatrix{(A , d_{\omega})\colon  \ 
A^0  \ar^(.65){d_{\omega}}[r] & A^1
\ar^(.5){d_{\omega}}[r] 
& A^2   \ar^(.5){d_{\omega}}[r]& \cdots },
\end{equation}
using as differentials the maps given by  
\begin{equation}
\label{eq:dw}
d_{\omega} (\alpha)= d \alpha  + 
\omega \cdot \alpha, 
\end{equation}
for all $\alpha \in A$.  
Computing the homology of these chain complexes 
for various values of the parameter $\omega$, and 
keeping track of the resulting Betti numbers  yields some interesting 
subsets of the affine space $H^1(A)$.
More precisely, for each $i$ and $r$, define  
\begin{equation}
\label{eq:rra}
\RR^i_r(A)= \{\omega \in H^1(A)   
\mid \dim_{\C} H^i(A, d_{\omega}) \geq r\}
\end{equation}
to be the $i$-th {\em resonance variety of depth} $r$ of the $\cdga$ $(A, d)$.  
It is readily seen that these sets are Zariski closed 
subsets of the affine space $H^1(A)$, for all $i\le q$.  
The set $\RR^i_1(A)$ is simply denoted by $\RR^i(A)$.

When the differential $d$ is zero, the resonance varieties 
of the algebra $A$ are homogeneous subsets of $H^1(A)=A^1$. 
In general, though, the resonance varieties of a $\cdga$ 
are not homogeneous. Here is a simple example, extracted 
from \cite{MPPS}. 

\begin{exm}
\label{ex:solv}
Let $A$ be the exterior algebra on generators $x,y$ in 
degree $1$, endowed with the differential given by $d{x}=0$ 
and $d{y}=y\wedge x$. Then $H^1(A)=\C$, generated by $x$, 
and $\RR^1(A)=\{0,1\}$.
\end{exm}

\section{Algebraic models and resonance}
\label{sect:models}

We now return to the topological setting from the Introduction. 
Throughout, $X$ will be a connected CW-complex 
with finite $q$-skeleton, for some $q\ge 1$
(for short, a $q$-finite space).  
There are two kinds of resonance varieties that one 
can associate to the space $X$, depending on which 
$\cdga$ one uses to approximate it.  

The most direct approach is to take the cohomology algebra 
$H^{\hdot}(X,\C)$, endowed with the zero differential. Let 
$\RR^i_r(X)$ be the resonance varieties of this $\cdga$. 
By the above discussion, these sets are Zariski closed, homogeneous 
subsets of the affine space $H^1(X,\C)$, for all $i\le q$.  
As before, we will denote $\RR^i_1(X)$ by $\RR^i(X)$. 

While relatively easy to compute, these varieties 
may not provide accurate enough information about our space, 
since the cohomology algebra $H^{\hdot}(X,\C)$ may not 
be a (rational homotopy) model for $X$. 

We thus turn to Sullivan's model of polynomial forms on $X$ 
(see \cite{Su}).   
This model, which we denote by $(\Omega^{\hdot}(X),d)$, is a 
$\cdga$ defined over $\Q$ which imitates the de Rham 
algebra of differential forms on a smooth manifold;  
in particular, $H^{\hdot}(\Omega(X))\cong H^{\hdot}(X,\C)$, 
as graded rings. 

The difficulty is that, in general, Sullivan's model does not have 
good finiteness properties. So let us assume $\Omega(X)$ 
has the same $q$-type as a $q$-finite $\cdga$ $(A,d)$. 
(As pointed out in \cite{DP-ccm, MPPS}, this assumption is 
satisfied in many situations of geometric interest.) In this case, the 
resonance varieties $\RR^i_r(A)$ are identified with Zariski-closed subsets 
of $H^1(X, \C)$, for all $i\le q$, since
\begin{equation}
\label{eq=ident}
H^1(A) \cong H^1(\Omega (X)) \cong H^1(X, \C) .
\end{equation} 

The next result provides a comparison between the two types 
of resonance varieties associated to our space $X$, under 
the above assumptions. 

\begin{theorem}[\cite{MPPS}]
\label{thm:mpps}
Suppose $X$ is $q$-finite and $\Omega^{\hdot}(X)$ has the same $q$-type as a 
$q$-finite $\cdga$ A. Then, for all $i\le q$, the tangent 
cone at $0$ to the resonance variety $\RR^i_r(A)$ is contained 
in $\RR^i_r(X)$.
\end{theorem}

As the next example shows, the inclusion from 
Theorem \ref{thm:mpps} can well be strict.  

\begin{exm}
\label{ex:heis}
Let $X$ be the $3$-dimensional Heisenberg nilmanifold.  Then 
$X$ is a circle bundle over the torus, with Euler number $1$; thus, 
$H^1(X,\C)=\C^2$ and all cup products of degree $1$ classes  
vanish.  It follows that $\RR^1(X)$ coincides with $H^1(X,\C)$.  

On the other hand, $X$ admits as a model the exterior 
algebra $A$ on generators $x,y,z$ in degree $1$, 
with differential $dx=dy=0$ and $dz=x\wedge y$. It is 
now a simple matter to check that $\RR^1(A)=\{0\}$, 
thereby proving the claim.
\end{exm}

\section{Characteristic varieties}
\label{sect:cv}

We now turn to another type of homological jump loci associated to 
our space $X$.  Let $\pi=\pi_1(X)$, and let $\Hom(\pi,\C^*)$ be 
the algebraic group of complex-valued, multiplicative characters 
on $\pi$, with identity $1$ corresponding to the trivial representation. 
For each character $\rho\colon \pi\to \C^*$, let 
$\C_{\rho}$ be the corresponding rank $1$ local system on $X$.  

The {\em characteristic varieties}\/ of $X$ are the 
jump loci for homology with coefficients in such local systems, 
\begin{equation}
\label{eq:cvx}
\VV^i_r(X)= \{\rho \in \Hom(\pi,\C^*) \mid \dim_{\C} H_i(X, \C_{\rho}) \geq r\}.
\end{equation}
For each $i\le q$, the sets $\VV^i_r(X)$  are Zariski-closed 
subsets of the character group $\Hom(\pi,\C^*)=H^1(X,\C^*)$. 
The set $\VV^i_1(X)$ is simply denoted by $\VV^i(X)$.

When the space $X$ has an algebraic model $A$ with good finiteness 
properties, the characteristic varieties of $X$ may be identified  
around the identity with the resonance varieties of $A$.  More 
precisely, we have the following basic result. 

\begin{theorem}[\cite{DP-ccm}]  
\label{thm:dp-ccm}
Assume $X$ is $q$-finite and $\Omega^{\hdot}(X)$ has the same 
$q$-type as a $q$-finite $\cdga$ A.  Then, for all $i\le q$ and all 
$r\ge 0$, the germ at $1$ of $ \VV^i_r (X)$ is isomorphic to
the germ at $0$ of $ \RR^i_r (A)$. Furthermore, all 
these isomorphisms are induced by an analytic
isomorphism $H^1(X,\C^*) _{(1)} \cong H^1(A)_{(0)}$   
obtained from the map $\exp_*\colon  
\Hom(\pi,\C)\to\Hom(\pi, \C^*)$. 
\end{theorem}

A precursor to this theorem can be found in the pioneering work 
of Green and Lazarsfeld \cite{GL87, GL91} on the cohomology 
jump loci of compact K\"{a}hler manifolds.   An important 
particular case (for $i=1$ and under a $1$-formality assumption) 
was first established in \cite{DPS}.  

\section{Alexander-type invariants}
\label{sect:alex inv}

As before, let $X$ be a connected CW-complex. 
Let $\nu \colon \pi \surj G$ be an epimorphism 
from the fundamental group of $X$ to an abelian group $G$, 
and let $X^{\nu}$ be the corresponding Galois cover.  The 
action of the group of deck-transformations on this cover   
puts a $\C[G]$-module structure on the homology groups 
$H_i(X^{\nu}, \C)$. We shall call these modules the {\em 
Alexander-type invariants}\/ of the cover $X^{\nu} \to X$. 

Assume now that the CW-complex $X$ has finite $q$-skeleton, 
for some $q\ge 1$. Since $X$ has finitely many $1$-cells, its  
fundamental group is finitely generated.  Thus, the quotient $G$ 
is a finitely-generated abelian group, and the group-algebra 
$R=\C[G]$ is a commutative, finitely generated $\C$-algebra.  Likewise, the 
Alexander-type invariants $H_i(X^{\nu}, \C)$ are finitely-generated 
$R$-modules, for all $i\le q$.  

\begin{theorem}[\cite{PS-plms}]  
\label{thm:ps-plms}
With notation as above, the following equality holds:
\begin{equation*}
\label{eq:ps-plms}
\nu^*\Big(\bigcup_{i\le q} \supp (H_i(X^{\nu}, \C))\Big)=
\im (\nu^*) \cap \Big(\bigcup_{i\le q} \VV^i (X)\Big),
\end{equation*}
where $\nu^*\colon \Hom(G,\C^*) \to \Hom(\pi,\C^*)$ is 
the monomorphism induced by $\nu$.
\end{theorem}

In particular, taking $\nu=\ab$, we see that the variety 
$\bigcup_{i\le q} \VV^i (X)$ coincides with the union up to degree $q$ of the 
support varieties of the Alexander invariants $H_i(X^{\ab}, \C)$.

\section{Completion and resonance}
\label{sect:main}

As above, let $X$ be a $q$-finite CW-complex.  Set $\pi=\pi_1(X)$, 
and let $\nu \colon \pi \surj G$ be an epimorphism to an 
abelian group $G$.  
Let $\nu^*\colon H^1(G,\C)\to  H^1(\pi, \C)$ be the induced homomorphism. 
Identifying $H^1(\pi,\C)$ with $H^1(X, \C)$, we may view the image of $\nu^*$ as a linear 
subspace of $H^1(X, \C)$. 

\begin{theorem}[\cite{PS-mrl}] 
\label{thm:ps}
With notation as above, the following implication holds:
\begin{equation*}
\label{eq:i1}
\im (\nu^*) \cap \bigg(\bigcup_{i\le q} \RR^i (X) \bigg) = \{ 0\} \imp
\dim \bigoplus_{i\le q}  \widehat{H}_{i} (X^{\nu}, \C) < \infty. 
\end{equation*}
\end{theorem}

We now sharpen this result, under the assumption that $X$ 
has a $q$-finite $q$-model $A$, i.e., $\Omega^{\hdot}(X)$ 
has the same $q$-type as $A$. 

\begin{theorem}
\label{thm1} 
Suppose that the $q$-finite CW-complex $X$ has a $q$-finite 
$q$-model $A$. Then the following 
conditions are equivalent.
\begin{enumerate}
\item \label{q1}
The complex vector space $\bigoplus_{i\le q} \widehat{H}_{i} (X^{\nu}, \C)$ 
is finite-dimensional.
\item \label{q2} 
The point $0$ is an isolated point in  
the variety $\im (\nu^*) \cap \big(\bigcup_{i\le q} \RR^i (A)\big)$.
\end{enumerate}
\end{theorem}

\begin{proof} 
By Proposition \ref{prop:isol}, the $\C$-vector space  
$\bigoplus_{i\le q} \widehat{H}_{i} (X^{\nu}, \C)$ is finite-dimensional 
if and only if the identity character $1\in \Hom (G, \C^{*})$ is 
an isolated point in the variety $\bigcup_{i\le q} \supp (H_i(X^{\nu}, \C))$.

By Theorem \ref{thm:ps-plms}, this variety may be identified with
the intersection of the algebraic subgroup 
$\im \big(\nu^*\colon \Hom(G,\C^*) \to \Hom(\pi,\C^*)\big)$ with  
the corresponding union of characteristic varieties, 
$\bigcup_{i\le q} \VV^i(X)$. 

Finally, by Theorem \ref{thm:dp-ccm}, the 
germ at $1$ of the above intersection may be 
identified with the germ at $0$ of the trace on 
the linear subspace 
$\im \big(\nu^* \colon H^1(G,\C)\to  H^1(\pi, \C)\big)$ 
of the corresponding union of resonance varieties,
$\bigcup_{i\le q} \RR^i (A)$. 

Putting these facts together completes the proof.
\end{proof}

Taking $\nu=\ab$ in the above result proves Theorem \ref{thm:intro} 
from the Introduction.

\section{Positive weights and formal spaces}
\label{sect:qp}

Let $A$ be a rationally defined \cdga.  We say that 
$A$ has {\em positive weights}\/ if $A^i=\bigoplus_{j\in \Z}
A^i_j$ for each $i\ge 0$, and, moreover, these
vector space decompositions are compatible
with the $\cdga$ structure and satisfy the condition
$A^1_j=0$, for all $j\le 0$. As we shall see in Section \ref{sec=gmodels}, 
the Gysin models of a connected quasi-projective manifold 
have positive weights.

The existence of positive weights on a $\cdga$ model $A$ for 
a space $X$  leads to a strong connection between 
the resonance varieties of $A$ and $X$. 

\begin{theorem}[\cite{MPPS}]
\label{thm:mpps-bis}
Assume  $X$ is $q$-finite, $\Omega^{\hdot}(X)$ has the same $q$-type as a 
$q$-finite $\cdga$ A with positive weights, and the identification \eqref{eq=ident} 
preserves $\Q$-structures. Then $\RR^i_r(A)\subseteq \RR^i_r(X)$, 
for all $i\le q$ and $r\ge 0$.
\end{theorem}

Under the above (more restrictive) assumptions, Theorem \ref{thm:ps}
also follows from Theorems \ref{thm1} and \ref{thm:mpps-bis}. 
The positive weight property also enters into the following result,
applicable to Gysin models of quasi-projective manifolds.

\begin{corollary}
\label{cor:fd} 
Under the assumptions from the above theorem, the $\C$-vector space 
$\bigoplus_{i\le q} \widehat{H}_{i} (X^{\ab}, \C)$ is finite-dimensional 
if and only if $\bigcup_{i\le q} \RR^i (A)= \{ 0\}$.
\end{corollary}

\begin{proof}
By Theorem \ref{thm1}, the dimension of 
$\bigoplus_{i\le q} \widehat{H}_{i} (X^{\ab}, \C)$ is finite 
if and only if $0$ is an isolated point in $\bigcup_{i\le q} \RR^i (A)$. Note that
the $\C^*$-action on $A^1$ associated to the positive weight decomposition 
leaves both $H^1(A)$ and the resonance varieties $\RR^i_r (A)$ invariant.
Since plainly the orbit $\C^* \cdot \alpha$ is positive-dimensional for 
$0\ne \alpha \in A^1$ and $0$ belongs to the closure of this orbit, 
our claim follows.
\end{proof}

A connected space $X$ is said to be {\em $q$-formal}\/ if Sullivan's 
model $\Omega^{\hdot}(X)$ has the same $q$-type as  
the cohomology algebra $H^{\hdot}(X,\C)$ endowed with the 
$0$ differential.  By \cite{DGMS}, 
compact K\"{a}hler manifolds are $\infty$-formal. 
For a detailed discussion of formality (especially 
$1$-formality) in our context, we refer to \cite{PS-formal}. 

Notice that, if $X$ is $q$-finite and $q$-formal, 
we may take $A=(H^{\hdot}(X,\C), d=0)$ as an 
appropriate $\cdga$ model  in Theorem \ref{thm:dp-ccm}. 

\begin{theorem}[\cite{DP-ann}] 
\label{thm:dp}
Suppose $X$ is $1$-finite and $1$-formal. Then the following holds: 
\begin{equation*}
\label{eq:i2}
\dim  \bigoplus_{i\le 1}  \widehat{H}_{i} (X^{\ab}, \C)< \infty  \same
\bigcup_{i\le 1} \RR^i (X)= \{ 0\}.
\end{equation*}
\end{theorem}

We may now generalize this result, as follows. 

\begin{theorem}
\label{thm:formal} 
Let $X$ be a $q$-formal CW-complex with finite $q$-skeleton. Then  
\begin{equation*}
\label{eq:equiv}
\dim \bigoplus_{i\le q}  \widehat{H}_{i} (X^{\nu}, \C) < \infty \same
\im (\nu^*) \cap \Big(\bigcup_{i\le q} \RR^i (X) \Big) = \{ 0\} ,
\end{equation*}
for any abelian Galois cover $X^{\nu}$ of $X$. 
\end{theorem}

\begin{proof}
By assumption, the cohomology algebra $A=H^{\hdot}(X,\C)$ endowed 
with the $0$ differential is a $q$-finite $q$-model for $X$.   Clearly, 
this $\cdga$ has positive weights: simply set $A^i_j=A^i$ if $j=i$, 
and zero otherwise. Furthermore, the $\C^*$-action on $A^1$ 
associated to this positive weight decomposition is just  
$\C^*$-multiplication. The desired conclusion follows in 
the same way as in Corollary \ref{cor:fd}, using the fact that 
each resonance variety $\RR^i(X)$ is a homogeneous variety. 
\end{proof}

\begin{exm}
\label{ex:heis-bis}
Let $X$ be the Heisenberg manifold from Example \ref{ex:heis}. 
This space is not $1$-formal, and $\RR^1(X)=\C^2$; 
thus, neither Theorem \ref{thm:dp} nor Theorem \ref{thm:ps} apply 
in this case.

On the other hand, as we saw previously, $X$ admits a finite  
model $(A,d)$ for which  $\RR^1(A)=\{0\}$.  
Thus, we may apply  Theorem \ref{thm1}   
to conclude that $\widehat{H}_1(X^{\ab}, \C)$ is finite-dimensional.   
(In fact, direct computation shows that $H_1(X^{\ab}, \C)=\C$.)
\end{exm}

\section{Gysin models}
\label{sec=gmodels}

Let $X$ be a connected quasi-projective manifold. Choose a smooth 
compactification, $X= \overline{X}\setminus D$, where $D=\bigcup_{j\in J} D_j$ 
is a finite union of smooth divisors with normal crossings.
There is then an associated rational $\cdga$, 
$(A^{\hdot}, d)=A^{\hdot}(\overline{X},D)$, called the {\em Gysin model} 
of the compactification, constructed as follows.

As a vector space, $A^k= \bigoplus_{p+l=k} A^{p,l}$, with 
\begin{equation}
\label{eq:apl}
A^{p,l}= \bigoplus_{\mid S \mid =l} H^p \Big(\bigcap_{i\in S} D_i, \Q\Big)(-l),
\end{equation}
where $S$ runs through the $l$-element subsets of $J$ 
and $(-l)$ denotes the Tate twist. The multiplication in $A$ is induced by 
the cup--product, and has the property that 
$A^{p,l} \cdot A^{p',l'} \subseteq A^{p+p',l+l'}$. The differential, 
$d\colon A^{p,l} \to A^{p+2, l-1}$, is defined by using the various 
Gysin maps coming from intersections of divisors. 

Morgan proved in \cite{M} that $\Omega^{\hdot}_{\Q} (X)$ 
has the same $\infty$-type as $(A^{\hdot}, d)$,
and thus $\Omega^{\hdot} (X)$ has the same $\infty$-type 
as $(A^{\hdot}, d) \otimes \C$; moreover, 
the induced homology isomorphisms preserve $\Q$-structures. 
The weight of $A^{p,l}$ is  by definition $p+2l$, and this clearly 
gives a positive-weight decomposition of $(A^{\hdot}, d)$. 

Note that $(H^{\hdot}(\overline{X}, \C), d=0)$ is a sub-$\cdga$ of 
$A^{\hdot}(\overline{X},D)\otimes \C$, much simpler than the whole Gysin model.
Unfortunately, this subalgebra does not give a model for $X$, in general.

We will need a more detailed description of $(A^{\hdot}, d)= A^{\hdot}(\overline{X},D)$
in low degrees.  Omitting the  coefficients for cohomology (they will be 
assumed to be either $\Q$ or $\C$ for the rest of this section), 
we have:
\begin{align}
A^0&= A^{0,0}=H^0(\overline{X})
\\ \notag
A^1&= A^{1,0}\oplus A^{0,1} =
H^1(\overline{X}) \oplus \bigoplus_{j \in J}H^0(D_j)
\\ \notag
A^2&= A^{2,0}\oplus A^{1,1} \oplus A^{0,2}=H^2(\overline{X}) 
\oplus \bigoplus_{j\in J}H^1(D_j)  \oplus 
\bigoplus_{\{ j,j' \}\subset J}H^0(D_j\cap D_{j'}),
\end{align}
with differential $d\colon A^0 \to A^1$ the zero map, 
and differential $d\colon A^1 \to A^2$ given by
\begin{equation}
\label{eq:diff}
d(\eta, (b_j)_{j \in J})=\Big(\sum_{j \in J}\iota_{j!}(b_j),0,0\Big)
\end{equation} 
for $\eta \in  H^1(\overline{X})$ and $b_j \in H^0(D_j)$. 
Here $\iota_j\colon D_j \to \overline{X}$ denotes the inclusion and 
$\iota_{j!}\colon H^0(D_j) \to H^2(\overline{X})$ the corresponding 
Gysin map. Note that $\iota_{j!} (1)\in H^2(\overline{X})$ is the Poincar\'{e} dual 
of the fundamental class $[D_j]$.
In analytic terms, $\iota_{j!} (1)= c_1(\OO_{\overline{X}}(D_j))$; 
see e.g.~\cite[p.~141]{GH}.

 To conclude 
this section, we provide the following cohomological criterion, 
as a warm-up exercise with Gysin models.

\begin{lemma}
\label{lem=iotaiso}
Let $\iota\colon X \to \overline{X}$ be the inclusion map, and 
assume that each divisor $D_j \subset \overline{X}$ 
is irreducible. Then the induced map 
$\iota^*\colon H^1(\overline{X}) \to H^1({X})$ is an isomorphism
if and only if the classes $\{ \iota_{j!} (1) \}_{j\in J}$ are linearly independent.
\end{lemma}

\begin{proof}
By functoriality of Gysin models \cite{M}, the map $\iota^*$ 
may be identified with the homomorphism induced on $H^1$ by the inclusion 
$(H^{\hdot}(\overline{X}), d=0) \inj A^{\hdot}(\overline{X},D)$.
By inspecting the definition \eqref{eq:diff} of the differential 
$d\colon A^1\to A^2$, we infer that the induced map
on $H^1$ is an isomorphism if and only if the restriction 
of $d$ to $A^{0,1}$ is injective.
\end{proof}
 
\section{Intersection forms and resonance}
\label{sect:non-proper}

Let $\overline{X}$ be a connected projective manifold of dimension $n\ge 2$, 
and let $D$ be a union of smooth divisors in $\overline{X}$. 
Let $\{D_j\}_{j \in J}$ be the irreducible components of $D$. 
We may define the intersection multiplicity $D_i \cdot D_j$ 
of two such components by the usual formula, 
\begin{equation}
\label{eq:ddot}
D_i \cdot D_j=\langle \eta_i \eta _j \al ^{n-2}, [\overline{X}] \rangle,
\end{equation}
where $\alpha$ is a K\"ahler form on $\overline{X}$ and 
$\eta_i\in H^2(\overline{X},\Z)$ is the cohomology class dual to  
$[D_i] \in H_{2n-2}(\overline{X},\Z)$. Alternatively,
using the projection formula from \cite[p.~11]{BPV}, 
we have that 
\begin{equation}
\label{eq:ddot-bis}
D_i \cdot D_j=\langle \iota_j^*(\eta_i \al ^{n-2}) , [D_j]\rangle,
\end{equation}
where $\iota_j\colon D_j \to \overline{X}$ is the inclusion.
When $n=2$ the choice of $\al$ is not necessary, since 
$D_i \cdot D_j$ coincides then with the usual intersection 
number of curves on a smooth surface.

Now let $(A^{\hdot}, d)=A^{\hdot}(\overline{X},D)$ be the 
Gysin model (over $\C$) associated to the divisor 
$D\subset \overline{X}$, as described in 
low degrees in the previous section. Also let 
$X=\overline{X}\setminus D$, and 
let $\iota\colon X \to \overline{X}$ be the inclusion map.
The decomposition of $D$ into  
connected components leads to a partition $J=J_1 \cup \cdots \cup J_m$.   
It follows from \eqref{eq:ddot} that the intersection matrix of $D$ splits
into blocks $I_1,\dots, I_m$, given by the intersection matrices of 
the divisors $D^k= \bigcup_{j\in J_k} D_j$.

\begin{lemma}
\label{lem=iotabis}
Suppose each of the intersection matrices $I_1,\dots, I_m$ is invertible. 
Then the map $\iota^*\colon H^1(\overline{X},\C) \to H^1(X,\C)$ is 
an isomorphism.
\end{lemma}

\begin{proof}
For each $j\in J$, set $\eta_j:=\iota_{j!} (1)\in H^2(\overline{X},\C)$. 
By Lemma \ref{lem=iotaiso}, we only need to check that the classes 
$\{  \eta_j \}_{j\in J}$ are independent. 

Suppose $\sum_{j\in J} b_j \eta_j =0$. Taking the cup product 
of $\sum_{j\in J} b_j \eta_j$ with the classes $\eta _i \al ^{n-2}$ 
for each $i\in J$ and evaluating on the fundamental class 
$[\overline{X}]$, we see using formula $\eqref{eq:ddot}$ that 
the vector $(b_j)_{j\in J}$ is in the kernel of the intersection matrix of $D$. 
By our hypothesis, this matrix is invertible.  
Hence, $b_j=0$ for all $j\in J$.
\end{proof}

Consider now a $1$-form $\omega \in H^1(\overline{X},\C)= H^1(A)$ and 
the associated  covariant derivative $d_\omega$ from \eqref{eq:dw}, given by 
$d_{\omega}(a)=  a\omega $ for $a \in A^0$ and
\begin{equation}
\label{eq:omega}
d_{\omega}(\eta, (b_j)_{j \in J})= \Big( \sum_{j \in J}\iota_{j!}(b_j)+\omega 
\wedge\eta, (b_j \iota_{j}^*(\omega) ) _{j \in J} ,0\Big)
\end{equation}
for $ (\eta, (b_j)_{j \in J}) \in A^1$. With a stronger hypothesis 
on the divisor $D$, we obtain the following lemma.

\begin{lemma}
\label{lem=final}
Suppose each of the intersection matrices $I_1,\dots, I_m$ is definite. 
If $d_{\omega}(\eta, (b_j)_{j \in J})= 0$, then $b_j=0$ for all $j \in J$.
\end{lemma}

\begin{proof}
Let $D^1,\dots,D^m$ be the connected components of $D$,  
and let $\{D_j\}_{j\in J_k}$ be the set of irreducible components 
of $D^k$. Let $J'_k$ be the set of indices $i\in J_k$ for  
which $\iota_i^*(\omega) = 0$, and set $J_k''=J_k\setminus J_k'$. 
The intersection matrix $I_k$ contains the 
diagonal block $I_k'$.   Since, by assumption, $I_k$ 
is a definite matrix, each of these blocks is an invertible matrix.  

If $i \in J''_k$, then the condition 
$d_{\omega}(\eta, (b_j)_{j \in J})= 0$  implies that $b_i=0$. 
If $i \in J'_k$, let us apply $\iota_i^*$ to the equality 
$\sum_{j \in J}\iota_{j!}(b_j)+\omega \wedge\eta=0$. 
Using formula \eqref{eq:ddot-bis}, we find that the vector 
$(b_j)_{j\in \bigcup_{k=1}^m J'_k}$ belongs to the kernel of
the invertible matrix $I'$ with blocks $I'_1, \dots,I'_m$.
\end{proof}

Observe now that the hypothesis of Theorem \ref{thm2:intro} 
coincides with that of Lemma \ref{lem=final}, and implies  
that of Lemma \ref{lem=iotabis}.  Thus, we may invoke 
these two lemmas in order to finish the proof of  the theorem, 
as shown next.

\begin{proof}
Recall that $X=\overline{X}\setminus D$, and 
$(A^{\hdot}, d)$ is the Gysin model (over $\C$) 
associated to the divisor $D\subset \overline{X}$. 
By Theorem \ref{thm:dp-ccm}, then, there is  
an analytic isomorphism $H^1(X,\C^*) _{(1)} \cong H^1(A)_{(0)}$ 
which identifies $ \VV^1_r (X)_{(1)}$ with $ \RR^1_r (A)_{(0)}$, for all $r\ge 0$.

On the other hand, the inclusion $(H^{\hdot}(\overline{X},\C), d=0) \inj (A^{\hdot}, d)$ 
identifies $H^{1}(\overline{X},\C)$ with $H^1(A)$, by Lemma \ref{lem=iotabis},
and $\RR^1_r (\overline{X})$ with $ \RR^1_r (A)$ for all $r\ge 0$, 
by Lemma \ref{lem=final}, and we are done.
\end{proof}

\begin{corollary}
\label{cor1} 
Suppose that the intersection matrices associated to the divisor 
$D=\overline{X}\setminus X$ are definite.
Then for any fixed projective curve $C_g$ of genus $g>1$, the 
inclusion $\iota\colon  X \to \overline{X}$ induces a bijection between the 
equivalence classes of fibrations $\overline{X}\to C_g$ (i.e., surjective morphisms 
with connected general fiber) and the equivalence classes of fibrations defined on $X$.
\end{corollary}

\begin{proof}
In view of the relationship between fibrations $X\to C_g$ 
and irreducible components of $\VV^1_1(X)_{(1)}$, 
and likewise for $\overline{X}$, established by Arapura in \cite{A},
the desired conclusion follows from Theorem \ref{thm2:intro}, in 
conjunction with with the identification 
$\VV^1_1 (\overline{X})_{(1)} \cong \RR^1_1 (\overline{X})_{(0)}$
from \cite{DPS}.
\end{proof}

\section{Examples and discussion}
\label{sect:ex}

In this last section, we make a few additional remarks and we 
examine several classes of examples related to Theorem \ref{thm2:intro}.

\begin{rem}
\label{rk1} 
There is an alternate way to prove Corollary \ref{cor1}, one which is both 
more direct, and also covers the cases $g=0$ and $g=1$.  
If $\dim \overline{X}=2$,  Zariski's 
Lemma (see Lemma (8.2) from \cite[p.~90]{BPV}), when  
coupled with our assumption on $D$  prevents any of the 
connected components $D^k$ of $D$ to 
be a fiber of a fibration $\overline{X}\to C_g$.
The general case can be reduced to the surface 
case by taking a general linear section.
\end{rem}

\begin{rem}
\label{rk2}  
It may happen that the map $\iota^*\colon H^1(\overline{X},\C) \to H^1(X,\C)$ is an 
isomorphism, yet the conclusion of  Theorem \ref{thm2:intro} does not hold.

For instance, take $\overline{X}=C_1\times C_1$ to be the product  
of two elliptic curves, and $D$ to be the diagonal in the product.  
In view of Lemma \ref{lem=iotaiso} (see also \cite[p.~64]{GH}), 
the map $\iota^*$ is an isomorphism. On the other hand, 
$\VV^1_1(X)$ is $2$-dimensional at $1$, by \cite[Example 10.2]{DPS}, 
yet $\RR^1_1(\overline{X})= \{ 0 \}$, by an easy computation.
\end{rem} 

\begin{exm} 
\label{ex:surf} 
Let $S$ be a normal, projective, connected complex  surface. 
Then the singular locus of $S$ is a finite set, say $\{ a_1,\ldots ,a_m \}$. 
Take $X$ to be the regular locus, $S_{\reg}=S \setminus \{a_1,\dots ,a_m\}$. 
By resolving each of these singularities, we obtain a surface $\overline{X}$ as 
in Theorem \ref{thm2:intro}: indeed, the corresponding matrices $I_k$ are 
all negative definite in this case, in view of the Mumford--Grauert criterion 
(see Theorem 2.1 from \cite[p.~72]{BPV}). The case $S=S_{\reg}$ 
was treated in \cite{DPS}.
\end{exm}

\begin{exm} 
\label{ex:whom} 
Fix a set of strictly positive weights $w=(w_1,w_2,w_3)$ and let 
$f\in \C[x,y,z]$ be a polynomial such that the affine surface $Y$ 
given by the equation $f=0$ 
in $\C^3$ has only isolated singularities, and the top degree part $f_d$ of 
$f$ with respect to the weights $w$ defines an isolated singularity at the 
origin of $\C^3$. 

Consider the closure $Z$ of $Y$ in the weighted projective 
space $\PP(w_1,w_2,w_3,1)$.
Then $Z=Y \cup D$, where $D$ is a smooth curve such that $D\cdot D>0$.
It follows as above that the surface $X=Y_{\reg}$, the smooth part of $Y$, 
satisfies the assumption of our Theorem \ref{thm2:intro}. The case of a weighted 
homogeneous polynomial (i.e., when $f=f_d$) was considered in \cite{DPSquasi}, 
where it was shown that such surfaces $X$ may not be $1$-formal.
\end{exm}

\begin{exm}
\label{ex:heis2}
Let $M$ be the $3$-dimensional Heisenberg nilmanifold discussed in 
Example \ref{ex:heis}, as well as in  \cite[Example 6.17]{DPS} and 
\cite[Example 8.6]{DPSquasi}. Since $M$ is an $S^1$-bundle with 
Euler number $1$, it has the same homotopy type as a $\C^*$-bundle 
$X$ over an elliptic curve $C=C_1$, associated to the line bundle
$L=\OO_C (s)$. We can construct a compactification $\overline{X}$ by taking 
the projective bundle $\PP(E)$ associated to the rank two vector bundle 
$E=L \oplus \OO_C$. Indeed, if $D_0$ and $D_{\infty}$ denote the 
divisors $\PP(0 \oplus \OO_C) \subset \PP(E)$ and $\PP(L \oplus 0) \subset \PP(E)$, 
then clearly $X= \overline{X} \setminus (D_0 \cup D_{\infty})$.

It is readily seen that $D_0^2=1$ and $D_{\infty}^2=-1$, which shows 
that we can apply Theorem \ref{thm2:intro}. Using the description 
of the cohomology of a projective bundle, it follows that 
$H^1(C,\C)=H^1(\overline{X},\C)$ and $H^2(C,\C)$ is a 
subspace of $H^2(\overline{X},\C)$, whence 
$\RR^1_1 (\overline{X})_{(0)}= \{ 0 \}$.  Theorem \ref{thm2:intro} 
now shows that there are no positive-dimensional components of 
$\VV^1_1(X)$ passing through the origin.

This behavior was predicted in Example \ref{ex:heis}, where it was 
pointed out that $\RR^1_1(A)=\{0\}$, for an explicit model $A$ of $M$. 
The Gysin model $A^{\hdot}(\overline{X},D)$ for $X$ can 
also be computed explicitly, but it is slightly more complicated than 
the model $A^{\hdot}$. For instance, the Hilbert series of  $A^{\hdot}$ is $(1+t)^3$,
whereas the Hilbert series of  $A^{\hdot}(\overline{X},D)$ is $(1+t)^4$. 
\end{exm}

\noindent\textbf{Acknowledgement}\textit{.} 
The research of A.D. was partially supported by Institut Universitaire de France.

The research of S.P. was partially supported by the Romanian Ministry of 
National Education, CNCS-UEFISCDI, grant PNII-ID-PCE-2012-4-0156.

The research of A.S. was partially supported by National Security 
Agency grant H98230-13-1-0225.

\smallskip\noindent 
Received: 

\salt 

\adresa{A. Dimca, Univ. Nice Sophia Antipolis, CNRS,  LJAD, UMR 7351,\\
 06100 Nice, France.\\
E-mail:  {\tt dimca\ap unice.fr}
}

\adresa{S. Papadima, Simion Stoilow Institute of Mathematics, \\
P.O. Box 1-764, RO-014700 Bucharest, Romania\\
E-mail:  {\tt Stefan.Papadima\ap imar.ro}
}

\adresa{A. Suciu, Department of Mathematics,\\
Northeastern University, Boston, MA 02115, USA\\
E-mail:  {\tt a.suciu\ap neu.edu}
}


\begin{thebibliography}{99} 

\bibitem{A} D.~A{\sc rapura}, 
\emph{Geometry of cohomology support loci for local systems  {\rm I}}, 
J. Algebraic Geom. \textbf{6} (1997), no.~3, 563--597.  

\bibitem{BPV}  W. B{\sc arth}, C. P{\sc eters}, 
A. V{\sc an} D{\sc e} V{\sc en}, 
\emph{Compact complex surfaces}, Ergeb. Math. Grenzgeb., vol.~4, 
Springer-Verlag, Berlin, 1984.

\bibitem{BW}  N.~B{\sc udur}, B.~W{\sc ang}, 
\emph{Cohomology jump loci of differential graded Lie algebras}, 
available at {\tt arXiv:1309.2264}.

\bibitem{DGMS}  P.~D{\sc eligne}, P.~{G\sc riffiths}, J
.~M{\sc organ}, D.~S{\sc ullivan},
\emph{Real homotopy theory of {K}\"{a}hler manifolds},
Invent. Math. \textbf{29} (1975), no.~3, 245--274.

\bibitem{DP-ccm} A.~D{\sc imca}, S.~P{\sc apadima},
\emph{Non-abelian cohomology jump loci from an analytic viewpoint}, 
Commun. Contemp. Math. \textbf{16} (2014), 
no.~4, 1350025 (47 p). 

\bibitem{DP-ann}  A.~D{\sc imca}, S.~P{\sc apadima},
\emph{Arithmetic group symmetry and finiteness properties 
of Torelli groups}, Annals of Math. \textbf{177} (2013), 
no.~2, 395--423.

\bibitem{DPS} A.~D{\sc imca}, S.~P{\sc apadima}, A.~S{\sc uciu},
\emph{Topology and geometry of cohomology jump loci}, 
Duke Math. Journal \textbf{148} (2009), no.~3, 405--457.

\bibitem{DPSquasi} A.~D{\sc imca}, S.~P{\sc apadima}, A.~S{\sc uciu},
\emph{Quasi-K\"{a}hler groups, $3$-manifold groups, and formality},  
Math. Z. \textbf{268} (2011), no. 1-2, 169--186.

\bibitem{Ei}  D.~E{\sc isenbud},
\emph{Commutative algebra with a view towards algebraic geometry},
Grad. Texts in Math., vol.~150, Springer-Verlag, New~York, 1995.

\bibitem{GL87} M.~G{\sc reen}, R.~L{\sc azarsfeld}, 
\emph{Deformation theory, generic vanishing theorems
and some conjectures of Enriques, Catanese and Beauville}, 
Invent. Math. \textbf{90} (1987), no.~2, 389--407.

\bibitem{GL91} M.~G{\sc reen}, R.~L{\sc azarsfeld},  
\emph{Higher obstructions to deforming cohomology groups 
of line bundles}, J. Amer. Math. Soc. \textbf{4} (1991), 
no.~1, 87--103.

\bibitem{GH}  P{\sc h}. G{\sc riffiths}, J. H{\sc arris}, 
\emph{Principles of algebraic geometry}, Wiley, New York, 1978.

\bibitem{MPPS}  A.~M{\sc \u{a}cinic}, S.~P{\sc apadima}, 
R.~P{\sc opescu},  A.~S{\sc uciu}, 
\emph{Flat connections and resonance varieties: 
from rank one to higher ranks}, {\tt arXiv:1312.1439}.                                                                                       

\bibitem{M}  J.~W.~M{\sc organ},
\emph{The algebraic topology of smooth algebraic varieties},
Inst. Hautes \'{E}tudes Sci. Publ. Math. \textbf{48} (1978), 137--204.

\bibitem{PS-formal} S.~P{\sc apadima}, A.~S{\sc uciu},
\emph{Geometric and algebraic aspects of $1$-formality}, 
Bull. Math. Soc. Sci. Math. Roumanie \textbf{52} (2009), 
no.~3, 355--375. 

\bibitem{PS-plms} S.~P{\sc apadima}, A.~S{\sc uciu},
\emph{Bieri--{N}eumann--{S}trebel--{R}enz invariants 
and homology jumping loci}, Proc. London Math. Soc. 
\textbf{100} (2010), no.~3, 795--834.

\bibitem{PS-mrl} S.~P{\sc apadima}, A.~S{\sc uciu},
\emph{Jump loci in the equivariant spectral sequence}, 
to appear in Math. Res. Lett., available at
{\tt arXiv:1302.4075}.

\bibitem{Su}  D.~S{\sc ullivan},
{\em Infinitesimal computations in topology}, Inst. Hautes
\'{E}tudes Sci. Publ. Math. \textbf{47} (1977), 269--331.

\end{thebibliography}
\end{document}